\documentclass[]{article}

\usepackage{amsmath,amssymb,latexsym,amsthm}

\newtheorem{thm}{Theorem}
\newtheorem{lemma}{Lemma}
\newtheorem{cor}{Corollary}
\newtheorem{defi}{Definition}
\newtheorem{prp}{Property}

\DeclareMathOperator{\sinc}{sinc}
\DeclareMathOperator{\res}{res}
\DeclareMathOperator{\pv}{p.v.}
\DeclareMathOperator{\im}{Im}
\title{An Identity Involving Integration with Respect\\to Variable Order of Fractional Derivative}
\author{Ivan Matychyn}
\date{}

\begin{document}

\maketitle

\section{Introduction}

In 1993, Samko and Ross \cite{SamRoss} introduced the study of fractional integration and differentiation when the order is not a constant but a function. This suggestion gave rise to a number of further ideas and results \cite{malinowska2015,sun2009,valerio2011}. In particular, this implies a possibility of integration with respect to derivative's order. Here an identity is presented, in which an expression involving Riemann--Liouville fractional derivative is integrated with respect to the derivative's order.

\section{Preliminaries}

\begin{defi}
Suppose $ f:[a,\infty)\to\mathbb{R}$ is an absolutely continuous function. The Riemann--Liouville (left-sided) fractional integral and derivative of order $\alpha$, $ m-1 < \alpha <m$, $ m\in\mathbb{N} $, are defined as follows:
\begin{align*}
J_{a+}^\alpha f(t) &= \frac{1}{\Gamma(\alpha)}\int_{a}^{t} (t-\tau)^{\alpha-1}f(\tau)d\tau,\ t> a,  \\
D_{a+}^\alpha f(t) &= \frac{d^m}{dt^m} J_{a+}^{m-\alpha} f(t),\quad t>a.
\end{align*}
In what follows we will omit the lower limit of integration in the notation if it is equal to zero, i.e. $ J^\alpha f(t) \triangleq J_{0+}^\alpha f(t) $, $ D^\alpha f(t) \triangleq D_{0+}^\alpha f(t)$.
\end{defi}

Here $ \Gamma(\alpha)$ denotes Euler's Gamma function.

\begin{prp}[Property 2.1 \cite{SamKilMar}] \label{prp1}
Let $ \alpha, \beta > 0 $. Then the following identity holds:
\begin{equation*}
D_a^\alpha (t-a)^{\beta-1}=\frac{\Gamma(\beta)}{\Gamma(\beta-\alpha)} (t-a)^{\beta - \alpha -1}.
\end{equation*}
\end{prp}

In particular, it follows from Proposition \ref{prp1}, that
\begin{align}
D^\alpha 1 &= \frac{t^{-\alpha}}{\Gamma(1-\alpha)}, \label{eq:1}\\
D^\alpha t^{n-1} &=\frac{\Gamma(n)}{\Gamma(n - \alpha)}t^{n - \alpha - 1}, \quad n\in \mathbb{N}. \label{eq:5}
\end{align}

\begin{defi}
The sinc function (``Cardinal Sine'') is defined as follows:
\begin{equation*}
\sinc (x) = 
\begin{cases}
\frac{\sin \pi x}{\pi x}, & x\ne 0,\\
1, & x=0.
\end{cases}
\end{equation*}
\end{defi}

\begin{prp} \label{prp1.5}
\begin{equation}
\int_{-\infty}^{\infty} \sinc(x) dx =1.
\end{equation}
\end{prp}

The following properties of Euler's Gamma function will be used in the sequel.

\begin{prp} \label{prp2}
\begin{equation} \label{eq:2}
\Gamma(1+\alpha)\Gamma(1-\alpha)=\alpha\Gamma(\alpha)\Gamma(1-\alpha)=\frac{\pi\alpha}{\sin(\pi\alpha)} = \frac{1}{\sinc(\alpha)},
\end{equation}
\end{prp}

\begin{prp} \label{prp3}
\begin{equation*}
\Gamma(z+n)  = (z)_n \Gamma(z),
\end{equation*}
where $(z)_n = z(z+1)\ldots (z+n-1)$ is the \textsl{Pochhammer symbol}.
\end{prp}

The following theorem  allowing to calculate some improper integrals with the help of contour integrals in the complex plane will also be used in the sequel.
\begin{thm}[Indented Trigonometric Integrals \cite{complex}] \label{contour}
Assume that $ P(z) $, $ Q(z) $, $ z\in \mathbb{C} $, are polynomials with real coefficients of degree  $ m $ and $ n $,  respectively, where $ n\ge m+1 $ and that  $ Q(z)$  has simple zeros at the points  $ t_1, \ldots, t_L $  on the $ x$-axis. If $p$ is a positive real number, and if  $ f(z)=\frac{e^{ipz}P(z)}{Q(z)}$, then we can compute the Cauchy Principal Value (P.V.) of the following integral
\begin{equation*}
\pv \int_{-\infty}^{\infty} \frac{P(x)}{Q(x)}\sin(px)dx=\im \left( 2\pi i \sum_{j=1}^{k} \res_{z=z_j} f(z)+  \pi i \sum_{j=1}^{L} \res_{z=t_j} f(z)\right).
\end{equation*}
\end{thm}

\section{The Main Result}

\begin{lemma} \label{lemma}
For every $ n=1,2,\ldots $ the following identity holds true:
\begin{equation}
\int_{-\infty}^{\infty}\frac{t^\alpha}{\Gamma(\alpha + 1)} (D^\alpha t^{n-1}) d\alpha = (2t)^{n-1}.
\end{equation}
\end{lemma}
\begin{proof}
For $ n=1$, in view of Properties \ref{prp2} and \ref{prp1.5}, we have  
\begin{equation*} 
\int_{-\infty}^{\infty} \frac{t^\alpha}{\Gamma(1+\alpha)} (D^\alpha 1) d\alpha=\int_{-\infty}^{\infty} \frac{\sin(\pi\alpha)}{\pi \alpha}d\alpha = \int_{-\infty}^{\infty} \sinc(\alpha) d\alpha = 1.
\end{equation*}

Now suppose $ n=2,3,\ldots $ Then, by virtue of Properties \ref{prp2} and \ref{prp3}, we have
\begin{equation*} 
\Gamma(n-\alpha)\Gamma(1+\alpha) = (1-\alpha)_{n-1} \Gamma(1-\alpha) \Gamma(1+\alpha)=(1-\alpha)_{n-1} \frac{\pi\alpha}{\sin(\pi\alpha)},
\end{equation*}
hence
\begin{equation} \label{eq:6}
\begin{aligned} 
\int_{-\infty}^{\infty} \frac{t^\alpha}{\Gamma(1+\alpha)} (D^\alpha t^{n-1}) d\alpha  
&= t^{n-1} \int_{-\infty}^{\infty} \frac{\Gamma(n)}{\Gamma(n - \alpha)\Gamma(1+\alpha)}d\alpha\\
&= t^{n-1} \int_{-\infty}^{\infty} \binom{n-1}{\alpha} d\alpha\\
&=\frac{t^{n-1}}{\pi}  (n-1)!\int_{-\infty}^{\infty} \frac{\sin(\pi\alpha)}{\alpha(1-\alpha)_{n-1}}d\alpha.
\end{aligned}
\end{equation}
It follows from Theorem \ref{contour} that
\begin{equation}
\begin{aligned}
\int_{-\infty}^{\infty} \frac{\sin(\pi\alpha)}{\alpha(1-\alpha)_{n}}d\alpha &= \int_{-\infty}^{\infty} \frac{\sin(\pi\alpha)}{\alpha(1-\alpha)(2-\alpha)\ldots (n-\alpha)}d\alpha \\
&= \im \left( \pi i \sum_{j=0}^{n} \res_{z=j} \frac{e^{i\pi z}}{z(1-z)(2-z)\ldots (n-z)} \right).
\end{aligned}
\end{equation}
Since
\begin{align*}
\res_{z=0} \frac{e^{i\pi z}}{z(1-z)(2-z)\ldots (n-z)} &= \frac{1}{n!},\\
\res_{z=j} \frac{e^{i\pi z}}{z(1-z)(2-z)\ldots (n-z)} &= \frac{1}{n!}\binom{n}{j},\ j=1,\ldots, n.
\end{align*}
we have
\begin{equation}
\int_{-\infty}^{\infty} \frac{\sin(\pi\alpha)}{\alpha(1-\alpha)_{n}}d\alpha = \frac{\pi}{n!}\sum_{j=0}^{n} \binom{n}{j}=\frac{\pi 2^n}{n!}.
\end{equation}
Thus
\begin{equation*}
\int_{-\infty}^{\infty} \frac{t^\alpha}{\Gamma(1+\alpha)} (D^\alpha t^{n-1}) d\alpha=\frac{t^{n-1}}{\pi}(n-1)!\frac{\pi 2^{n-1}}{{n-1}!}=(2t)^{n-1}.
\end{equation*}
\end{proof}

\begin{cor}
\begin{equation*}
\int_{-\infty}^{\infty} \binom{n}{\alpha} d\alpha = 2^n,\quad n\in\mathbb{N}.
\end{equation*}
\end{cor}
This Corollary follows from Lemma \ref{lemma} and \eqref{eq:6}.

\begin{cor}[Main Identity]
For any function $ f(t)$ that is analytic in some neighborhood of zero $(-\varepsilon,\varepsilon)$, $ \varepsilon>0 $, the following identity holds true
\begin{equation}
\int_{-\infty}^{\infty}\frac{t^\alpha}{\Gamma(\alpha + 1)} [D^\alpha f(t/2)]d\alpha = f(t),\quad t\in (-\varepsilon,\varepsilon).
\end{equation}
\end{cor}

\bibliographystyle{unsrt}

\begin{thebibliography}{1}

\bibitem{SamRoss}
Stefan~G Samko and Bertram Ross.
\newblock Integration and differentiation to a variable fractional order.
\newblock {\em Integral Transforms and Special Functions}, 1(4):277--300, 1993.

\bibitem{malinowska2015}
Agnieszka~B Malinowska, Tatiana Odzijewicz, and Delfim~FM Torres.
\newblock {\em Advanced methods in the fractional calculus of variations}.
\newblock Springer, 2015.

\bibitem{sun2009}
HongGuang Sun, Wen Chen, and YangQuan Chen.
\newblock Variable-order fractional differential operators in anomalous
  diffusion modeling.
\newblock {\em Physica A: Statistical Mechanics and its Applications},
  388(21):4586--4592, 2009.

\bibitem{valerio2011}
Duarte Val{\'e}rio and Jose~Sa Da~Costa.
\newblock Variable-order fractional derivatives and their numerical
  approximations.
\newblock {\em Signal Processing}, 91(3):470--483, 2011.

\bibitem{SamKilMar}
S.G. Samko, A.A. Kilbas, and O.I. Marichev.
\newblock {\em Fractional Integrals and Derivatives}.
\newblock Gordon \& Breach, Amsterdam, 1993.

\bibitem{complex}
John~H Mathews and Russell~W Howell.
\newblock {\em Complex analysis for mathematics and engineering}.
\newblock Jones \& Bartlett Publishers, 2012.

\end{thebibliography}

\end{document}